\documentclass[12pt]{article}
\usepackage{amsmath,amscd,amsthm,amsfonts,amssymb}

	\usepackage{amsmath,amscd,amsthm,amsfonts,amssymb} 
	\usepackage{pgfpages}

	\usepackage{adjustbox}
	\usepackage{graphicx}
	\usepackage{tikz}
	\usepackage{caption}
	\usepackage{enumerate}% http://ctan.org/pkg/enumerate
	\begin{document}
	\newtheorem{fact}{Fact}
	\newtheorem*{ass*}{Main assumption}
	\newtheorem{thm}{Theorem}
	\newtheorem{cor}{Corollary}
	\newtheorem{lem}{Lemma}
	\newtheorem{slem}{Sublemma}
	\newtheorem{prop}{Proposition}
	\newtheorem{defn}{Definition}
	\newtheorem{conj}{Conjecture}
	\newtheorem*{ques*}{Question}
	\newtheorem{claim}{Claim}
	\newcounter{constant} 
	\newtheorem{rmk}{Remark}
	\newcommand{\newconstant}[1]{\refstepcounter{constant}\label{#1}}
	\newcommand{\useconstant}[1]{\Delta_{\ref{#1}}}
	\title{Sequences of surfaces in $4$-manifolds}
	\date{}
\author{Marina Ville}

\maketitle
\begin{abstract} Let $(\Sigma_n)$ be a sequence of surfaces immersed in a 
$4$-manifold $M$ which converges to a branched surface $\Sigma_0$ .\\
We denote by $k^T_p$ (resp. $k^N_p$) the amount of curvature of the tangent bundles
$T\Sigma_n$ (resp. normal bundles $N\Sigma_n$) which concentrates around a branch
point $p$ of $\Sigma_0$ when $n$ goes to infinity. Alternatively $k^T\pm k^N$ measures how much the twistor degrees drop when we go from $\Sigma_n$ to $\Sigma_0$. For complex algebraic curves, $k^T+k^N=0$..\\
In some instances - 1) if $\Sigma_0$ is made up of at most $3$ branched disks or 2) if $\Sigma_0$ is area minimizing or 3) if the $\Sigma_n$'s are minimal - we show that
$-k^T\geq |k^N|$  and we investigate the equality case.\\
When the second fundamental forms of the $\Sigma_n$'s have a common
$L^2$ bound, we relate $k^T$ and $k^N$ to the bubbling-off of a current $C$ in the 
Grassmannian $G_2^+(M)$. If the $\Sigma_n$'s are minimal, $C$ is a complex curve.
\end{abstract}
\section{Introduction - Motivation}
\subsection{Statement of the problem}
Consider an algebraic function $F(z_1,z_2)$ defined in a small ball around $(0,0)$ in $\mathbb{C}^2$. If $\epsilon$ is small, assume that the curves $F^{-1}(\epsilon)$ are smooth and converge to a singular curve $F^{-1}(0)$, made of branched disks. We derive the first Betti number of the $F^{-1}(\epsilon)$'s by computing  the
{\it Milnor number} ([Mi], see also [Ru 1])  on the Puiseux coefficients on 
$F^{-1}(0)$.  Very roughly speaking: going from the smooth curves to the singular one,  we lose in topology but we gain in 
singularity  and we know exactly how much topology we have lost just by looking at the singular curve. 
\begin{ques*}What remains of this nice picture if $(\Sigma_n)$ is a
sequence
of $2$-surfaces embedded in a $4$-manifold which
degenerates to a branched surface  $\Sigma_0$ ? Can we define 
a Milnor number in this context?
\end{ques*}
We recall branched surfaces in \ref{definition of branched immersions}  and we state what convergence we require in \ref{paragraphe avec l'assumption}.\\
\\
Complex curves in K\"ahler surfaces are a special case of
minimal surfaces (Wirtinger's theorem) so this question makes particular sense if the $\Sigma_n$'s are minimal surfaces.\\
\\
The question of generalizing the Milnor number to the non-complex 
algebraic case has been around for some time. We would like to mention the work of R\'emi Langevin (see [La] for example) and of Lee Rudolph:  in 
particular [Ru 2] which  contains a 
construction closely related to ours.
\subsection{Some results}
We consider a Riemannian $4$-manifold $M$ and a sequence of immersed/embedded surfaces $\Sigma_n$ in $M$ which converges to a surface $\Sigma_0$ with a branch points $p$.  
 We look at the amount of curvature of the tangent bundles $T\Sigma_n$ and normal bundles $N\Sigma_n$'s which bubbles off when the $\Sigma_n$'s converge to $\Sigma_0$. \\
 $\bullet$.  We let
 \begin{equation}
 \Sigma_n^\epsilon=B(p,\epsilon)\cap\Sigma_n
 \end{equation}
 where $B(p,\epsilon)$ is the ball centered at $p$ of radius $\epsilon$ w.r.t. the distance defined by the Euclidean metric defined by extending the metric on $T_pM$ in a small neighbourhood of $p$.\\ 
 $\bullet$ We denote by $K_n^T$ and $K_n^N$ the  curvatures  of $T\Sigma_n$ and $N\Sigma_n$, and let
\begin{equation}
k^T_p=\frac{1}{2\pi}\lim_{\epsilon\longrightarrow 0}\lim_{n\longrightarrow \infty}
\int_{\Sigma_n^\epsilon}K^T_n
\end{equation}
\begin{equation}
k^N_p=\frac{1}{2\pi}\lim_{\epsilon\longrightarrow 0}\lim_{n\longrightarrow \infty}
\int_{\Sigma_n^\epsilon}K^N_n
\end{equation}
These definitions are local and they do not depend on the metric.
NB. If we change the orientation on $M$, $k^N$ is changed in  $-k^N$.\\
We call $k^T$ and $k^N$ the {\it tangent and normal fallouts} and give topological interpretations of these numbers.
\begin{prop}\label{formule pour le tangent fallout}
	If there are $m$ branched disks $\mathcal D_i$, $i=1,...,m$ of branching orders $N_i-1$,
	\begin{equation}
	\label{gauss-bonnet2}
	k^T=2-\sum_{i=1}^m (N_i+1)-2\lim_{\epsilon\longrightarrow 0}\lim_{n\longrightarrow \infty}
	g(\Sigma_n^\epsilon)
	\end{equation}
\end{prop}\begin{prop}\label{formule pour le normal fallout}
	Assume that the $\Sigma_n$'s only have transverse double points in a neighbourhood of $p$. For each branched disk ${\mathcal D}_i$, $i=1,...,m$, making up $\Sigma_0$, we pick a vector non zero vector $X_i$  transverse to the plane tangent to  ${\mathcal D}_i$ at $p$. For $\epsilon$ small enough, the orthogonal projections of the $X_i$'s to $\mathbb{S}^3$ give us a framing $\hat{X}$ of $\Gamma^\epsilon=\Sigma_0\cap \mathbb{S}(0,\epsilon)$. We let $sl_X(\Gamma^\epsilon)$ be the self-linking number of $\Gamma^\epsilon$ w.r.t. $\hat{X}$; then for $\epsilon$ small enough, 
	\begin{equation}\label{kn}
	k^N=sl_X(\Gamma^\epsilon)-2\lim_{\epsilon\longrightarrow 0}\lim_{n\longrightarrow \infty}(\#\mbox{double points of\ }\Sigma_n^\epsilon)
	\end{equation}
	In the case of a single branch point of branching order $N-1$, $\Gamma^\epsilon$ is naturally presented as an $N$-braid and $sl_X(\Gamma^\epsilon)$ is equal to the algebraic length, or writhe of this braid.
	\end{prop}
In the complex algebraic case, $K^T+K^N=0$, thus
$
k^N+k^T=0
$ and $k^N=\mu+r-1$ where $\mu$ is the Milnor number of the singularity and $r$ is the number of boundary components of $\Sigma_n\cap \mathbb{S}(p,\epsilon)$.\\
If the $\Sigma_n$'s are symplectic or superminimal surfaces, we also have $|k^N|=-k^T$ (\S \ref{paragraph on equality} below).
But in the general case, we only have partial information.
\begin{thm}\label{proposition: inegalite} Assume that  {\bf one of the following is true}
\begin{enumerate}[1)]
	\item 
	The $\Sigma_n$'s are minimal surfaces
	\item 
	$\partial \Sigma_0$ can be presented as a braid
	\item
	$\Sigma_0$ is area minimzing (a special case of 2)
\end{enumerate}
\begin{equation}\label{inegalite principale}
\mbox{Then} \ \ \ \ \  \ \ \ \ \  \ \ \ \ \ |k^N|\leq -k^T
\end{equation}

\end{thm}
\subsection{Sketch of the paper}
In \S \ref{Definitions and main properties}) we give definitions and derive the topological descriptions (Propositions \ref{formule pour le normal fallout} and \ref{formule pour le normal fallout}) of $k^T$ and $k^N$.
1) of Theorem \ref{proposition: inegalite} 1.  follows immediately from the formulae for the curvatures of minimal surfaces (see \S \ref{preuve des surfaces minimales}). In \S \ref{section on braids} we explain and develop 2) and 3) of Theorem \ref{proposition: inegalite}.\\
In \S \ref{section: second fundamental form}, we assume that 
the second fundamental
forms of the $\Sigma_n$'s have a common $L^2$ bound.  Then  the areas of the lifts of the $\Sigma_n$'s in the
Grassmannian $G_2^+(M)$ of oriented $2$-planes tangent to $M$ have a
common upper bound; when $n$ goes to infinity,  a closed $2$-current $C$  
bubbles off in the Grassmannian $G_2^+(T_pM)$ and we can read
$k^T$ 
and $k^N$ off the homology class of
$C$. If the $\Sigma_n$'s are minimal surfaces, $C$ is a complex curve.\\
 In \S \ref{paragraph on equality} we show that equality in (\ref{inegalite principale}) of Theorem \ref{proposition: inegalite} can have strong implications: for example if the $\Sigma_n$'s are minimal, $-k^T=|k^N|$ implies that $\partial\Sigma_0$ is a quasipositive link.  In \S \ref{exemples et contre-exemples} we give examples where $|k^N|< -k^T$.
\subsection{A motivation: the twistor degrees}\label{section:twistor degree}
The twistor degree was first defined by Eells and Salamon ([E-S]); some authors call it the {\it
	adjunction number}. If $\Sigma$ is a closed oriented surface immersed in a compact oriented Riemannian $4$-manifold $M$  with tangent bundle $T\Sigma$ and normal bundle $N\Sigma$, the positive and negative twistor degrees of $\Sigma$ are
\begin{equation}\label{positive twistor degree} 
d_+(\Sigma)=c_1(T\Sigma)+c_1(N\Sigma)\ \ \ \ \ \ \ \ \   d_-(\Sigma)=c_1(T\Sigma)-c_1(N\Sigma)
\end{equation}
\begin{rmk} ([G-O-R], see \S \ref{definition of branched immersions}). If $\Sigma$ has branch points, 
	the bundles $T\Sigma$ and $N\Sigma$ are also well defined and so are the $d_{\pm}(\Sigma)$'s.
\end{rmk} 
If $\Sigma$ is a complex curve in a complex surface $X$, the complex structure induces orientations on $\Sigma$ and $X$ and for these orientations, the adjunction formula ([G-H]) tells us that \begin{equation}\label{adjunction formula}
d_+(\Sigma)=<c_1(X),[\Sigma]>
\end{equation}
where $[\Sigma]$ is the $2$-homology class of $\Sigma$ in $X$ and $<.,.>$ denotes the duality between homology and cohomology. Thus, in the complex case, the twistor degree is a homotopy invariant. This is a very strong property: in the general case,  it is only an isotopy invariant between non branched immersions.\\
So a question arose in the 1990s: if a sequence of minimal surfaces $(\Sigma_n)$ converges to a branched minimal surface $\Sigma_0$, how do the $d_{\pm}(\Sigma_n)$'s and  $d_{\pm}(\Sigma_0)$ compare? We recover from 1) of Theorem \ref{proposition: inegalite} 
\begin{thm}\label{semicontinuity}([C-T])
	Let $M$ be a $4$-manifold and let $(\Sigma_n)$ be a sequence of immersed minimal surfaces in $M$ converging to a branched minimal surface $\Sigma_0$. Then 
	\begin{equation}\label{equation:semicontinuity}
	d_+(\Sigma_n)\leq d_+(\Sigma_0)\ \ \ \ \ \ d_-(\Sigma_n)\leq d_-(\Sigma_0)
	\end{equation}
\end{thm} 
 A trivial but key observation is:
\begin{rmk}\label{changement de signe}
	If we change the orientation of $M$ but not on $\Sigma$, $c_1(N\Sigma)$ is changed into $-c_1(N\Sigma)$.
\end{rmk}
Hence if we change the orientation on $M$, $d_+$ becomes
$d_-$ and vice versa. Thus, even if the $\Sigma_n$'s are area minimizing, we may not have equality in (\ref{equation:semicontinuity}). For exemple \footnote{EXEMPLE.  Take $C_n=\{[z_0,z_1,z_2]\in\mathbb{C}P^2\slash z_0z_1^2+z_2^3=\frac{1}{n}z_0^3\}$,\\ $
	C_0=\{[z_0,z_1,z_2]\in\mathbb{C}P^2\slash z_0z_1^2+z_2^3=0\}$. then  (cf.  [G-H] for  details) for $n>0$,
	$c_1(TC_n)+c_1(NC_n)=-9$ while $c_1(TC_0)+ c_1(NC_0)=-3$.
	Yet, at the branch point $k^N-k^T=0$.} if $M$ is the projective plane $\mathbb{C}P^2$ endowed with the orientation {\it opposite} to the standard one, complex curves in $\mathbb{C}P^2$ are area minimizing surfaces in $M$. Any sequence $(C_n)$ of smooth algebraic curves of $\mathbb{C}P^2$ converging to a branched one $C_0$ verifies in $M$ 
$$d_+(C_n)< d_+(C_0)$$
yet at each branch point, we have $k^T=k^N$, thus $|k^N|=-k^T$ as in $\mathbb{C}P^2$.\\
\\ 
This is a reason to replace the global study of the twistor degree by the local study of the normal and tangent fallouts.\\
\\
REMARK.  To illustrate the difference between the local and global points of view, see  [M-S-V] for a sequence $(\Sigma_n)$ of minimal closed tori in a $(\mathbb{S}^4,g)$ converging to a sphere $\Sigma_0$  with two transverse double points. For one double point, $k^N=k^T$ and for the other one $k^N=-k^T$. So we have $|k^N|=-k^T$ in both cases but $d_+(\Sigma_n)<d_+(\Sigma_0)$ and 
$d_-(\Sigma_n)<d_-(\Sigma_0)$. \\
\\
We derive from 3) of Theorem \ref{proposition: inegalite} 
\begin{cor}\label{twistor degree and area minimizing}
	Let $M$ be an oriented compact Riemannian $4$-manifold, let $\Sigma_0$ be an area minimzing branched surface in $M$ and let $(\Sigma_n)$ be a sequence of compact oriented surfaces embedded in $M$ which converge to $\Sigma_0$ for the Hausdorff distance and uniformely on every compact set not containing the singular points of $\Sigma_0$. Then, for $n$ large enough, 
	$$d_+(\Sigma_n)\leq d_+(\Sigma_0)\ \ \ \ \ \ \ \ \ 
	d_-(\Sigma_n)\leq d_-(\Sigma_0)$$
\end{cor} 
\subsection{Concluding question}
The following remains open 
\begin{ques*}
	Are there examples with $|k^N|>- k^T$?
\end{ques*}

\subsection*{Acknowledgements}First I would like to recall the memory of Alexander
Reznikov whose questions and conversations in the 1990s were the starting point of this paper. Sadly
he will never read it. Many wonderful conversations with the regretted Jim Eells enlightened and stimulated me.\\
A crucial step in this work was a meeting 
with R{\'e}mi Langevin who explained to me material from his
thesis and later pointed out a mistake in an earlier draft.  In a single conversation Alex Suciu 
helped me leap forward. Daniel Meyer's careful remarks greatly helped me improve the exposition. Finally I am very grateful to Marc Soret who has been an amazing partner for the study of branch points for many years. 
\section{Definitions and main properties}\label{Definitions and main properties}
\subsection{Branched immersions}\label{definition of branched immersions}
A map $f:S\longrightarrow M$ from a Riemann surface 
$S$ to a manifold $M$ is a {\it branched immersion}  if it is an immersion everywhere except at a discrete set of points
called
{\it branch points} and locally parametrized  ([G-O-R]) by a complex not necessarily holomorphic variable $z$ in a small disk $D$ centered at $0$ in $\mathbb{C}$ as
\begin{equation}\label{the branched immersion f}
f:z\mapsto (Re(z^N)+o_1(|z|^N), Im(z^N)+o_1(|z|^N), o_1(|z|^N),o_1(|z|^N))
\end{equation}
 where a function is a  $o_1(|z|^N)$ 
	if it is a $o(|z|^N)$ and its first partial derivatives are 
	$o(|z|^{N-1})$'s.
\\
If $G_2^+(M)$ is the
Grassmannian of oriented $2$-planes in $M$, the Gauss map
 sends a regular point $q$ of $f$ in $D$ to the  tangent plane to $f(D)$ at $q$; it extends continuously across the branch points ([G-O-R], [Gau]) so it defines an oriented $2$-plane bundle above $S$, called the {\it image tangent bundle}
$Tf$; and by taking orthogonal complements,  a normal
$2$-plane bundle $Nf$.
\subsection{Main assumption}\label{paragraphe avec l'assumption}
Throughout the paper, we consider the following situation.
\begin{ass*}\label{basic assumption}
$M$ is an oriented Riemannian $4$-manifold,
\begin{itemize}
	\item 
 the $\Sigma_n$'s are smooth $2$-surfaces in $M$ embedded or immersed with transverse double points
 \item 
 the $\Sigma_n$'s are connected and their genera have a common upper bound  
 \item 
 the areas of the $\Sigma_n$'s have a common upper bound
 \item
  $\Sigma_0$ is a finite union of disks topologically embedded in $M$ which are all branched at the same point $p\in M$. 
\end{itemize}
The $\Sigma_n$'s converge to $\Sigma_0$ in the Hausdorff sense and uniformely smoothly on every compact subset of $M$ not containing $p$.
\end{ass*}
\subsection{Curvature formulae}\label{curvature formulae}
Let $\Sigma$ be a surface in a Riemannian $4$-manifold $M$, we denote by $R^M$ be the curvature of $M$ and by $B$ the second fundamental form of $\Sigma$. If $p$ is a point in $\Sigma$, we let $(e_1,e_2)$ (resp. $(e_3,e_4)$) be an orthonormal basis of $T_p\Sigma$ (resp. $N_p\Sigma$). We have
\begin{equation}
	\Omega^T(e_1, e_2)=-\|B(e_1,e_2)\|^2+<B(e_1,e_1),B(e_2,e_2)>
	+<R^M(e_1, e_2)e_1, e_2>
\end{equation}
\begin{equation}\label{tangent to minimal}
	=-\|B(e_1,e_2)\|^2-\|B(e_1,e_1)\|^2+<R^M(e_1,e_2)e_3,e_4>\ \ \ \mbox{if\ }\Sigma \mbox{\ is minimal}
\end{equation}
\begin{equation}
	\Omega^N(e_1, e_2)=(B(e_1, e_1)-B(e_2, e_2))\wedge B(e_1,e_2)
	+<R^M(e_1,e_2)e_3,e_4>
\end{equation}
\begin{equation}\label{normal to minimal}
	=2B(e_1, e_1)\wedge B(e_1,e_2)
	+<R^M(e_1,e_2)e_3,e_4>\ \ \ \mbox{if\ }\Sigma \mbox{\ is minimal}
\end{equation}
Note that in (\ref{normal to minimal}),  we have identified $\Lambda^2(N_p\Sigma)$ with $\mathbb{R}$.
Since we are integrating on smaller and smaller surfaces, the terms in $R^M$ in (\ref{tangent to minimal}) and (\ref{normal to minimal}) have no bearing on the computation of $k^N$ and $k^T$. Moreover, since $B$ is a tensor, we can also compute $k^T$ and $k^N$ using the Euclidean metric obtained by extending the metric on $T_p\Sigma$.

\subsection{The tangent fallout: proof of Proposition \ref{formule pour le tangent fallout}} For $\epsilon$ and $n$, the Gauss-Bonnet formula with boundary states
\begin{equation}\label{gaussb}
\int_{\Sigma_n^\epsilon}\Omega^T_n-2\pi\chi(\Sigma_n^\epsilon)= 
-\int_{\partial\Sigma_n^\epsilon}k_g
\end{equation}
where $k_g$ is the geodesic curvature of the curve 
$\partial\Sigma_n^\epsilon$ on the surface $\Sigma_n^\epsilon$.\\
When $n$ tends to infinity, the right-hand side in (\ref{gaussb}) tends to 
$$-\int_{\partial\Sigma_0^\epsilon}k_g.$$
where $k_g$ is the geodesic curvature of $\partial \Sigma_0^\epsilon$ 
inside $\Sigma_0^\epsilon$. 
Since $\partial\Sigma_0^\epsilon$ is asymptotic to $m$ circles of multiplicity $N_i$,
$$\lim_{\epsilon\longrightarrow 0}\int_{\partial\Sigma_0^\epsilon}k_g=2\pi \sum_{i=1}^mN_i$$  
\qed
\subsection{The normal fallout: proof of Proposition 
\ref{formule pour le normal fallout}}
\subsubsection{Framings}
A framing of a link $L$ in $\mathbb{S}^3$ is a vector field along $L$ tangent to $\mathbb{S}^3$ and nowhere tangent to $L$; a framing gives us a self-linking number  of $L$. For our present purpose, we slightly rewrite things and define a framing $X$ of $L$ to be a vector field $X$ in $\mathbb{R}^4$ along $L$  and nowhere tangent to $L$. The the orthogonal projection $\hat{X}$ is never tangent to $L$ and we define the self-linking number $sl_X(L)$ to be the linking number between $L$ and the link obtained by pushing $L$ slightly in the direction of $\hat{X}$.\\
Notice that, if $X_t$, $t\in[0,1]$ is a smooth family of framings of $L$ in $\mathbb{R}^4$, $sl_{X_1}(L)=sl_{X_2}(L)$ 
\subsubsection{A lemma}
\begin{lem}\label{lemme sur stokes}
	Let $\Sigma$ be a surface with boundary and let $F:L\longrightarrow
	\Sigma$ be a $U(1)$-bundle. For $L$, we denote  $\langle,\rangle$ its scalar product, $J$ its complex structure, $\nabla$ its connection and $\Omega$ its curvature. \\
	If $s$ is a section of $L$
	which vanishes nowhere on the boundary of $\Sigma$, we define the following form on $\partial\Sigma$:
	$$\omega(u)=\frac{<\nabla_u s, Js>}{\|s\|^2}=<\nabla_u(\frac{s}{\|s\|}),
	J(\frac{s}{\|s\|})>.$$
	$$\mbox{Then}\  \ \ \ \ \ \ \frac{1}{2\pi}\int_\Sigma\Omega=\frac{1}{2\pi}\int_{\partial\Sigma}\omega+
	\sum_{i=1}^m\mbox{index}(z_i)$$
	where the $z_i$'s $i=1,...,m$ are the zeroes of $s$ inside $\Sigma$.
\end{lem}
\begin{proof}
	We notice that $\omega$ is the restriction to $\partial \Sigma$  of a connection form on $L$ defined on $\Sigma$ outside of the zeroes of $s$; thus $\Omega=d\omega$ and we use Stokes' formula.
\end{proof}
\subsection{The proof}
 % Take a small enough collar ${\mathcal C}=\partial\Sigma_n^\epsilon\times [0,1]$ with   $\partial\Sigma_n^\epsilon$ identified with $\partial\Sigma_n^\epsilon\times \{0\}$. In this collar, we deform $X$ into  $X_n$ which we require to belong to $N\Sigma_n^\epsilon$ above $\partial\Sigma_n^\epsilon\times \{1\}$;%
 First we note that for $\epsilon$ small enough and $n$ large enough, $X$ is never parallel to the direction of a point in $q\in\partial\Sigma_n$ and it does not belong to  $T_q\Sigma_n$. \\
For $n$ large enough and $\epsilon$ small enough, there is  an isotopy between the following two framings of  $\partial \Sigma_n^\epsilon$: $\hat{X}$ and the framing $X_n$ obtained by projecting $X$ to 
$N\Sigma_n^\epsilon\cap T_q\mathbb{S}(0,\epsilon)$ for  $q\in\partial \Sigma_n^\epsilon$.
Note that the normal bundles $N\Sigma_n^\epsilon$'s are taken w.r.t. the metric $g$ while the orthogonal projections are w.r.t. the Euclidean metric $g_0$ on $T_pM$.
Since the framings are isotopic, we have
\begin{equation}
sl_{X_n}(\Gamma_\epsilon)=sl_{X}(\Gamma_\epsilon)
\end{equation}
Then extend $X_n$ as a global section, also denoted $X_n$, of $N\Sigma_n^\epsilon$ above $\Sigma_n^\epsilon$. \\ Let $\nabla^{(n)}$ be the connection on $N\Sigma^\epsilon_n$ derived by the Levi-Civita connection on $M$; let $J_n$ the complex structure on $N\Sigma_n$ compatible with its  $SO(2)$-structure. We define the form $\omega_n^N$ by
\begin{equation}
\label{grosse formule1}
\forall u\in T(\partial \Sigma_n),\ \  \omega_n^N(u)=
\frac{<\nabla^{(n)}_uX_n, J_nX_n>}{\|X_n^N\|}
\end{equation}
Lemma \ref{lemme sur stokes} tells us that 

\begin{equation}
\label{stokes sur normal}
\frac{1}{2\pi}\int_{\Sigma_n^\epsilon} \Omega_n^N= 
\frac{1}{2\pi}\int_{\partial \Sigma_n^\epsilon}\omega_n^N
+N(X_n,\Sigma_n^\epsilon)
\end{equation} 
where $N(X_n,\Sigma_n^\epsilon)$ the number of zeroes 
of $X_n$ in  $\Sigma_n^\epsilon$. \\
We now push $\Sigma_n^\epsilon$ slightly in the direction of $X_n$ and get  surface $\hat{\Sigma}_n^\epsilon$; then, a zero of $X_n$ corresponds to an intersection of $\Sigma_n^\epsilon$ with   $\hat{\Sigma}_n^\epsilon$. The other points in $\Sigma_n^\epsilon\cap\hat{\Sigma}_n^\epsilon$ come from the double points of $\Sigma_n^\epsilon$, each double point giving rise to $2$ points in $\Sigma_n^\epsilon\cap\hat{\Sigma}_n^\epsilon$. All the points in this discussion are counted  with sign. \\
Now the total number of points in $\Sigma_n^\epsilon\cap\hat{\Sigma}_n^\epsilon$ is equal  to $lk(\Gamma^\epsilon,\hat{\Gamma}^\epsilon)$,where $\hat{\Gamma}^\epsilon$ is obtained by pushing $\Gamma^\epsilon$ slightly in the direction of $\hat{X}$  so we derive
\begin{equation}
N(X_n,\Sigma_n^\epsilon)=lk(\Gamma^\epsilon,\hat{\Gamma}^\epsilon)-2(\#\mbox{double points of\ }\Sigma_n^\epsilon)
\end{equation}
The components of $\partial\Sigma_n^\epsilon$ tend to great circles in $\mathbb{S}(p,\epsilon)$ bounding flat disks in $\mathbb{R}^4$. The various $X_n$'s tend to vectors which are parallel along these disks and the limits of the $J_n$'s are also parallel. Thus,
for a fixed $\epsilon$, we have
\begin{equation}\label{limite des omega}\lim_{\epsilon\longrightarrow 0}
\lim_{n\longrightarrow\infty}
\int_{\partial \Sigma_n^\epsilon}\omega_n^N=\lim_{\epsilon\longrightarrow 0}
\int_{\partial \Sigma_0^\epsilon}\omega_0^N=0
\end{equation}
\qed
\section{Minimal surfaces: proof of Theorem \ref{proposition: inegalite}.1}\label{preuve des surfaces minimales}
  Theorem \ref{proposition: inegalite} 1) follows immediately from \S \ref{curvature formulae}.\\
\\
EXEMPLE. Consider a branched minimal disk $\Sigma_0$ parametrized as follows
$$z\mapsto (Re(z^3)+o(|z|^3), Im(z^3)+o(|z|^3), Im(z^{50}), Re(z^{110}e^{i\alpha}))$$
where $\alpha$ is a generic real number which ensures that $\Sigma_0$ is topologically embedded. It follows from [S-V2] that the writhe of $\partial\Sigma_0$ is $20$, up to sign. Thus, if $(\Sigma_n)$ is a sequence of connected minimal surfaces converging to $\Sigma_0$ as in Assumption 1, we have $g(\Sigma_n)\geq 9$
for $n$ large enough.

\section{Braids: proof of Theorem \ref{proposition: inegalite}.2 and \ref{proposition: inegalite}.3}\label{section on braids}
\subsection{Preliminaries on braids}
We refer the reader to [B-B] for more details.\\
If $D$ is an oriented line in $\mathbb{R}^3$, a {\it closed braid} $L$ of axis $D$  in $\mathbb{R}^3$ is a disjoint union of a loops $\gamma_i(t)$, $i=1,...,k$ whose cylindrical coordinates $(\rho_i, \theta_i, z)$, with $D$ as a vertical axis,  verify for every $t$,
\begin{equation}
\label{cylindrique}
\rho_i(t)\neq 0, \ \ \ \ \theta'_i(t)>0
\end{equation}
The {\it number of strands} of $L$ is sum of the degrees of $\theta_i:\mathbb{S}^1\longrightarrow :\mathbb{S}^1$. The {\it writhe} or {\it algebraic length} $e(L)$ is the linking number of $L$ with a loop $\hat{L}$ obtained by pushing $L$ slightly in the direction of $D$.\\
If we add a point at infinity $\infty$, the $L$ above becomes a braid in the $3$-sphere, its axis being the great circle defined by $D$ and $\infty$. Conversely consider a link $L$ in $\mathbb{S}^3$, a great circle $C$ disjoint from $L$ and a point $A$ in $C$. Identify $A$ with the point at infinity on $\mathbb{S}^3$ and let $D=C-\{A\}$. The link $L$ will be a braid of axis $C$ if the components of its stereographic projection with pole $A$ satisfies (\ref{cylindrique}) for cylindrical coordinates of axis $D$.
%If $\Sigma_0$ is a branched disk of degree $N$ with tangent plane at the origin $P$, $\partial\Sigma_0^\epsilon$ is naturally presented as an $N$-braid (see [Vi4], [S-V1]). Its axis is a great circle in a plane orthogonal to $P$ (or if we take $\epsilon$ small enough, a plane transverse to $P$). Going to the $3$-sphere, we have
\\
\\
 If $L$ is an oriented link in $\Bbb{S}^3$ we let $\chi_s(L)$ be the greatest Euler characteristic of a smooth $2$-surface $F$ in $\Bbb{B}^4$ without closed components smoothly embedded in $\Bbb{B}^4$ with boundary $L$. 
Lee Rudolph proved the following {\it slice Bennequin inequality}.
\begin{thm}\label{slice bennequin}([Ru 3]) 
	Let $\beta$ be a closed braid with $n$ strands and algebraic crossing number 
	$e(\beta)$. Then
	$$\chi_s(L)\leq n-e(\beta).$$
\end{thm}
\subsection{A criterion for the boundary to be a braid}
\begin{thm}\label{quand le bord est une tresse}
	1) Assume that $\Sigma_0$ is made up of $m$ branched disks with oriented tangent planes at $p$ denoted $P_1$, ..., $P_m$, each one of branching order $N_i-1$. If there exists an oriented plane $Q$ meeting all the $P_i$'s transversally at $p$ in a positive intersection, then $\partial\Sigma_0$ is a braid with $N_1+...+N_m$ strands and algebraic length $k^N$. Thus 
	\begin{equation}\label{l'inegalite cruciale}
	|k^N|\leq -k^T
	\end{equation}
	2) This is true in particular if one of the following three assumptions is true: 
	\begin{enumerate}[(i)]
		\item 
		there exists an $a\in\{1,...,m\}$ such that for every 
		$i\in\{1,...,m\}$, $i\neq a$, 
		$$P_a\cdot P_i>0$$
		\item 
		$m\leq 3$ 
		\item 
		there exists an orthogonal complex structure $J$ on $T_pM$ w.r.t. which every $P_i$ is a complex line. This happens in particular if $\Sigma_0$ is area minimizing (see [Mo]).
	\end{enumerate}
	\end{thm}
\begin{proof} 1). We parametrize one of the branched disks by  $f:D\longrightarrow \mathbb{R}^4$  as in (\ref{the branched immersion f}),  $N-1$ being its branching order and we let $P$ be its tangent plane at $p$. Let $e_1,e_2$ be a positive basis of $P$ and $e_3,e_4$ be a positive orthonormal basis of $Q$. Identify $e_4$ with $\infty$; let $S$ be the stereographic projection of $\mathbb{S}^3$  with pole $e_4$ to the $3$-space $\mathbb{R}^3$ generated by $e_1,e_2,e_3$. The loop $\partial (f(D)\cap\mathbb{S}(0,\epsilon)))$ is close to the circle of radius $\epsilon$ in the plane  generated by $e_1,e_2$ travelled $N$ times. Thus the stereographic projection  $S(\partial (f(D)\cap\mathbb{S}(0,\epsilon))$ is a $N$-braid in $\mathbb{R}^3$ of axis $e_3$. We do the same for every $P_i$ and derive that $\partial\Sigma_0$ is a braid $L$ with $N_1+...+N_m$ strands and of axis $e_3$.\\
	The vector $e_3$ verifies the assumptions for the vector $X$ in Proposition \ref{formule pour le normal fallout} so $k^N=e(L)$.
	 Thus Theorem \ref{slice bennequin} yields 
	\begin{equation}
\label{rudolph applique au k}
k^T=\chi(\Sigma_n^\epsilon)-\sum_{i=1}^mN_i\leq e(\partial\Sigma_0)= k^N
\end{equation}
 We now change the orientation on $\mathbb{R}^4$; the quantity $k^N$ becomes $-k^N$ while $k^T$ is unchanged. We define the oriented plane $\tilde{Q}$ as $Q$ with the opposite orientation. The plane $\tilde{Q}$ meets each $P_i$ positively for the new orientation on $\mathbb{R}^4$ so (\ref{rudolph applique au k}) applies and we have $k^T\leq -k^N$.
\\
\\
	2) If (i) is true, there exists a plane $Q$ close to $P_a$ such that $Q\cdot P_a=1$ and $Q\cdot P_i=1$ for every $i$ with $i\neq a$ so 1) applies.\\ To prove (ii) one checks that, given $3$ two by two transverse planes, there is necessarily one which intersects the other two positively and we use (i).\\ Property (iii) follows from the fact that two complex lines always intersect positively for the orientation on $M$ given by the complex structure. If $\Sigma_0$ is area minimizing, the tangent cone at $p$, i.e. the tangent planes to the branched disks making up $\Sigma_0$ is area minimizing in $T_pM$ (see [Ch]); hence these planes are all complex for some parallel complex structure ([Mo]).
\end{proof}
REMARK. Not all boundaries of a branched surface can be presented as a closed braid. For exemple, take two disks $\mathcal D_1$, $\mathcal D_2$ branched at $p$ and suppose that $T_p\mathcal D_1$ and $T_p\mathcal D_2$ are equal set-wise but have opposite orientation; then the boundary of $\Sigma=\mathcal D_1\cup \mathcal D_2$ cannot be presented as a braid. 
\section{Bounding the second fundamental form}\label{section: second fundamental form}
We look at sequence of surfaces satisfying Assumption \ref{basic assumption} with $L^2$ bounded second fundamental forms. As a counterexemple, consider a sequence of disk-like surfaces $(S_n)$ in $\mathbb{R}^3$ which converges to $S_0=\{(x,y,0)\in\mathbb{R}^3: x^2+y^2\leq 1\}$ as in Assumption \ref{basic assumption} and such that for every $n$, $S_n$ has $n$ bumps of height $1$ in $\{(x,y,z)\in\mathbb{R}^3: x^2+y^2\leq \frac{1}{n}\}$. The $L^2$ norm of the second fundamental forms of the $S_n$'s tend to infinity.
\subsection{Preliminaries: twistor spaces}
If $M$ is a Riemannian $4$-manifold, we let $\Lambda^2(M)$ be the bundle of tangent $2$-vectors and  $\star:\Lambda^2(M)\longrightarrow \Lambda^2(M)$ be the Hodge star operator which allows the splitting of $\Lambda^2(M)$ into  $\pm 1$-eigenspaces $\Lambda^\pm (M)$. We denote by $Z_{\pm}(M)=\mathbb{S}(\Lambda^\pm (M))$ the unit sphere bundles of  $\Lambda^\pm (M)$.\\
We identify an oriented $2$-plane generated $P$ by a positive orthonormal basis $(e_1,e_2)$ with the $2$-vector $e_1\wedge e_2$ and derive the isomorphism
\begin{equation}\label{identification de la grassmannienne} 
G_2^+(M)\cong Z_+(M)\times Z_-(M)
\end{equation}
\begin{equation}
P\mapsto \big(\frac{1}{\sqrt{2}}(P+\star P),\frac{1}{\sqrt{2}}(P-\star P)\big)
\end{equation}
The bundle $Z_+(M)$ (resp. $Z_-(M)$) is the bundle of almost complex structures on $M$ compatible with the metric and which preserve (resp. reverse) the orientation on $M$. The $Z_\pm(M)$'s  are called {\it twistor spaces} and Eells-Salamon ([E-S]) endowed them with almost complex structures ${\mathcal I}_\pm$. We recall the construction of ${\mathcal I}_+$, the construction for ${\mathcal I}_-$ is identical.\\
Let $p\in M$ and $J$  an element in the fibre of $Z_+(M)$ above $p$. Since  $Z_+(M)$ inherits a metric and a connection from  $M$,  we split $T_{p,J}Z_+(M)$ into a vertical space $V_{p,J}Z_+(M)$ and a horizontal space $H_{p,J}Z_+(M)$.
\begin{itemize}
	\item 
	Since $H_{p,J}Z_+(M)$ is isomorphic to $T_pM$, we  define ${\mathcal I}_+$ on it by transporting the structure $J$ from $T_pM$.
	\item 
	The space $V_{p,J}Z_+(M)$ is the space orthogonal to $J$ in $\Lambda^+(M)_p$. It inherits a metric and orientation from  $\Lambda^+(M)_p$; these yield a complex structure on $V_{p,J}Z_+(M)$ and we define ${\mathcal I}_+$ as the opposite of this complex structure.
\end{itemize}
\begin{thm}([E-S])\label{theorem: eells-salamon}
	Let $\Sigma$ be a Riemann surface, let $M$ be a Riemannian
	$4$-manifold, and
	let $f:\Sigma\longrightarrow M$ be a conformal harmonic map. The lifts
	$$\tilde{f}:\Sigma\longrightarrow Z_{\pm}(M)$$
	are pseudo-holomorphic for  the almost complex structures
	${\mathcal J}_+$ and
	${\mathcal J}_-$.
\end{thm}
\begin{cor}
\label{eells-salamon sur la grassmannienne}
Putting together 
	${\mathcal J}^+$ and
	${\mathcal J}^-$ defines an almost complex structure ${\mathcal J}$ on $G_2^+(M)$; the lift of a minimal surface in $M$ to $G_2^+(M)$ is a  ${\mathcal J}$-holomorphic curve.
\end{cor}
The following proposition shows the connection between the $Z_\pm(M)$'s and the quantities $k^T$ and $k^N$. 
\begin{prop}([E-S], [Vi1])\label{liens avec les fibres}
	Let $\Sigma$ be a $2$-surface immersed in   an oriented Riemannian $4$-manifold $M$ and let $T_\pm:\Sigma\longrightarrow Z_\pm(M)$ be its twistor lift.  
	$$
	-c_1(T_+^\star VZ_+)=c_1(T\Sigma)+ c_1(N\Sigma)\ \ \mbox{and} \ 
	-c_1(T_-^\star VZ_-)=c_1(T\Sigma)-c_1(N\Sigma)
$$
\end{prop} 
\subsection{Bubbling off of a current}
Let $\Sigma$ be a surface in $M$ with second fundamental form $B$. Let $p\in\Sigma$ and let $e_1,e_2$ (resp. $e_3,e_4$) be a positive orthonormal basis of $T_p\Sigma$ (resp. $N_p\Sigma$). The tangent space to $G_2^+(M)$ at $e_1\wedge e_2$ is generated by $e_3\wedge e_4$ and $e_1\wedge e_j$, $e_2\wedge e_j$ for $j=3,4$. If $u$ is a unit vector in $T_p\Sigma$, and $j=3,4$
$$|\langle \nabla_u(e_1\wedge e_2), e_1\wedge e_j\rangle|=|\langle \nabla_u e_2,e_j\rangle|=|\langle B(u, e_2),e_j \rangle|\leq \|B\|.$$
Similarly $|\langle \nabla_u(e_1\wedge e_2), e_2\wedge e_j\rangle|\leq \|B\|$ and finally
$\langle \nabla_u(e_1\wedge e_2), e_3\wedge e_4\rangle=0.$
Thus bounds for the area of $\Sigma$ and for the $L^2$ norm of $B$ give us a bound for the area of the lift of $\Sigma$ in $G_2^+(M)$. We derive 
\begin{thm}\label{theorem: current}
	Let $M$, $(\Sigma_n)$, $\Sigma_0$ verifying  Assumption \ref{basic assumption}. Suppose  that 
	the $\Sigma^\epsilon_n$'s have common bounds for the area and for the $L^2$ norm of the second fundamental form. For $\epsilon,n$, we let 
	$\check{\Sigma}_n^\epsilon$ be the lift in  $G_2^+(M)$ of 
	$\Sigma_n^\epsilon$.\\
1)	There exists a closed $2$-current $C$ in $G_2^+(T_pM)$  such that:\\
	for every $\epsilon>0$, the sequence  $(\check{\Sigma}_n^\epsilon)$
	converges in the sense of currents and
	$$\lim_{n\longrightarrow\infty}\check{\Sigma}_n^\epsilon
	=\check{\Sigma}_0^\epsilon+C.$$
	2) We denote by $s_\pm$ the $2$-homology class of $G_2^+(\mathbb{R}^4)$ corresponding to the factor $Z_{\pm}(M)$ in $G_2^+(T_pM)=Z_+(M)_p\times Z_-(M)_p$ (cf. (\ref{identification de la grassmannienne}) above). Then 
	\begin{equation}
-[C]=(k^T+k^N)s_++(k^T-k^N)s_-
	\end{equation}
\end{thm}
\begin{proof}
1)	We define $C$ as the limit of $\hat{\Sigma}_n^\epsilon-\hat{\Sigma}_0^\epsilon$; it exists because of the bounds on the area and it is closed because $\partial(\hat{\Sigma}_n^\epsilon-\hat{\Sigma}_0^\epsilon)$ tend to $0$ as $n$ tends to infinity.\\
2) follows from  Proposition \ref{liens avec les fibres}. 
\end{proof}
\subsection{Minimal surfaces}
If the $\Sigma_n$'s are minimal, the current $C$ is actually a complex curve.
Corollary \ref{eells-salamon sur la grassmannienne} enables us to restate  
Theorem \ref{theorem: current}.
\begin{thm}\label{theorem: current in the minimal case}
	Let $M$, $(\Sigma_n)$ and $\Sigma_0$ be as in Assumption \ref{basic assumption} and suppose moreover that the $\Sigma_n$'s are minimal. There exists a
	complex curve $S$ cohomologous to the current $C$ of Theorem \ref{theorem: current} such that,
	for every $\epsilon>0$ small enough, 
	$$\lim_{n\longrightarrow\infty}\tilde{\Sigma}_n^\epsilon
	=\tilde{\Sigma}_0^\epsilon\cup S$$
	where the limit means: convergence of
	pseudo-holomorphic curves with boundary
	in the (Gromov) sense of cusp-curves.
In particular, the $\tilde{\Sigma}_n^\epsilon$'s converge to 
	$\tilde{\Sigma}_0^\epsilon\cup S$ in the Hausdorff sense; if the $\tilde{\Sigma}_n^\epsilon$'s are connected, then 
	$\tilde{\Sigma}_0^\epsilon\cup S$ is also connected.
\end{thm}
\begin{proof} The $\tilde{\Sigma}_n^\epsilon$'s are
pseudo-holomorphic curves with boundary in the pseudo-Hermitian manifold with boundary $G_2^+(M\cap\overline{\mathbb{B}}(p,\epsilon))$ ; their areas and genera are bounded. They satisfy the assumptions of Theorem 1 of [I-S] so they converges in the Gromov sense to a pseudo-holomorphic curve $S^\epsilon$. Since the $\Sigma_n$'s converge uniformely smoothly on compact sets not containing $p$, $S^\epsilon$ coincides with $\tilde{\Sigma}_0^\epsilon$ above $M-\{p\}$.
\end{proof} 
\begin{ques*} We prove the existence of the complex curve $S$ in $\mathbb{C}P^1\times\mathbb{C}P^1$ but we end up only using the information of its homology class; is it possible to use finer information on $S$ to shed light the convergence of the $\Sigma_n$'s?
	\end{ques*}
   
\section{The equality case $|k^N|=-k^T$}\label{paragraph on equality}
Here are two generalisations of complex curves where we still have equality between the fallouts.
\subsection{Superminimal surfaces}
Superminimal surfaces are the closest Riemannian analogue to complex curves in K\"ahler surfaces (see [Gau] for details) and their branch points are $C^1$ equivalent to branch points of a complex curve in $\mathbb{C}^2$ ([Vi2]).\\
A possibly branched surface $\Sigma$ immersed  in an oriented Riemannian $4$-manifold $M$ is  {\it right superminimal} (resp.
{\it left superminimal}) if its lift $J_+$ (resp. $J_-$) in $Z_+(M)$
(resp. $Z_-(M)$) is parallel w.r.t the connection induced by the Levi-Civita connection on $M$. Equivalently, the second fundamental form $B$ of $\Sigma$ is linear w.r.t. $J_+$ (resp. $J_-$) so the formulae in \S \ref{preuve des surfaces minimales} tell us that $\Omega^T+\Omega^N=0$ (resp. $\Omega^T-\Omega^N=0$). This proves 
\begin{prop}
	Let  $M$, $(\Sigma_n)$, $\Sigma_0$ and $p$ be as in Assumption \ref{basic assumption} and suppose 
	that the $\Sigma_n$'s are right (resp. left) superminimal. 
	Then 
	$$k^T=-k^N\ \ \ (resp.\ k^T=k^N).$$
\end{prop} 
\subsection{Symplectic curves}
\begin{prop}\label{autres cas d'égalite}
	Let $(\Sigma_n)$, $\Sigma_0$ be as in Assumption \ref{basic assumption} and assume that the $\Sigma_n$'s are symplectic for a symplectic structure $\omega$ in a neighbourhood of $p$.\\ If 
	$\omega\wedge \omega>0$ (resp. $\omega\wedge \omega<0$), then
	$$k^N+k^T=0\ \ \ \ \ \mbox{(resp.}\ \ \ \ \ k^N-k^T=0)$$
	
\end{prop}
\begin{proof}
	Darboux's theorem yields the existence of local coordinates $(x_1,y_1,x_2,y_2)$ such that $\omega=\sum_{i=1,2}dx_i\wedge dy_i$. The quantities $k^N,k^T$ do not depend on the metric so we compute them for a local  metric where $(\frac{\partial}{\partial x_1},\frac{\partial}{\partial y_1}, \frac{\partial}{\partial x_2},\frac{\partial}{\partial y_2})$ is an orthonormal basis.  We define a parallel complex structure $H$ by setting $H(\frac{\partial}{\partial x_i})=\frac{\partial}{\partial y_i}$; $H$ belongs to $Z_+$ (resp. $Z_-$) if $\omega\wedge \omega>0$ (resp. $\omega\wedge \omega<0$). For a $q$ in $\Sigma_n$, if we set $H(q)$ to be the North Pole in $(Z_{\pm})_q$, the lift of $q$ in $\tilde{\Sigma}_n^\epsilon$ is in the upper half-sphere containing $H(q)$. Hence $\tilde{\Sigma}_n^\epsilon$ is homotopically equivalent to $\{H(q)\slash q\in\Sigma_n\}$ and the limiting current $C$ is homologically trivial.
\end{proof}

We now investigate the equality cases in Theorem \ref{proposition: inegalite} and show that they have topological implications. 
\subsection{Equality case for braids}
\begin{prop}\label{premiere proposition sur l'egalite} Assume that 
	\begin{enumerate}
		\item $L=\partial\Sigma_0$ is presented as a braid (cf. Theorem \ref{quand le bord est une tresse})
		\item $|k^N|=k^T$
	\end{enumerate}
	Then, for $n$ large enough, $\epsilon$ small enough, $$\chi(\Sigma_n^\epsilon)=\chi_s(L)$$
	where $\chi_s(L)$ is the largest Euler characteristic of a surface in $\mathbb{B}^4$ bounded by $L$. 
\end{prop}  
\begin{proof}
	We apply Rudolph's theorem: $\chi_s(L)\leq n(L)-e(L)=
	\chi(\Sigma_n^\epsilon)\leq \chi_s(L)$
\end{proof}
EXEMPLE. In [S-V1], [S-V2], we studied branch points of the form
\begin{equation}\label{disque}
z\mapsto (Re(z^N)+o(|z|^N),Im(z^N)+o(|z|^N), Re(z^p)+o(|z|^q), Im((e^{i\alpha}z^q)+o(|z|^q) )
\end{equation}
for $N>q,p$ with $(N,p)=(N,q)=(p,q)=1$ and $\alpha$ a real number. They are bounded by {\it ribbon knots}; such a knot $K$ verifies $\chi_s(K)=1$. Thus, if $(\Sigma_n)$ converges to a $\Sigma_0$ of the form (\ref{disque}) and $-k^T=|k^N|$, the $\Sigma_n$'s have to be disks.
\subsection{Equality case for minimal surfaces}
\subsubsection{Preliminaries: quasipositive braids and surfaces}

\begin{defn}
	Let $\sigma_1,...,\sigma_{n-1}$ be generators of the braid group $B_n$, for some integer $n$. A $n$-braid $\beta$ is quasipositive if it can be written as
	\begin{equation}
	\beta=\prod_{k=1}^p\gamma_k\sigma_{i_k}
	\gamma_k^{-1}
	\end{equation}
	where the $\gamma_k$'s are $n$-braids.\\
	A link which can be represented by a quasi-positive braid is called a quasipositive link. 
\end{defn}

By definition, a {\it quasipositive surface} $F$ in a bidisk $\mathbb{D}_1\times\mathbb{D}_2$ has a  projection  $p_1:(z,w)\mapsto z$  which is a {\it simple} branched covering with no branch points on the boundary and preserves the orientation, except possibly at the branch points. The other projection $p_2:(z,w)\mapsto w$ preserves the orientation in a neighbourhood of the branch points of $p_1$. The link $L$ which bounds a quasipositive surface $\Sigma$  is quasipositive and moreover
\begin{equation}
\chi(\Sigma)=\chi_s(L)
\end{equation}
{\it Via} a diffeomorphism which smoothes the corners of the bidisk, this definition extends to surfaces in $\mathbb{B}^4$. Boileau-Orevkov ([B-O])  proved a very interesting result which we now state in a more restricted context. \\
Let $J_0$ be the complex structure on $\mathbb{C}^2$ and let  $\omega$ be the K\"ahler form $\mathbb{B}^4$ defined by $\omega(X,Y)=<X,J_0Y>$ ($<,>$ denotes the scalar product). A surface in $\mathbb{B}^4$ is {\it symplectic} if $\omega|_F>0$.\\
On $\mathbb{S}^3=\partial\mathbb{B}^4$, we define the contact form  $\xi$: 
$$\mbox{if\ }p\in\mathbb{S}^3,\ X\in T_p \mathbb{S}^3\ \ \xi(p)=<X,J_0p>$$
A loop $L$ in $\mathbb{S}^3$ is {\it ascending} w.r.t. the contact structure if $\xi|_L>0$. 
\begin{thm}\label{bo}([B-O])
	Let $(F,\partial F)\subset (\mathbb{B}^4,\mathbb{S}^3)$ be a smooth, oriented, properly embedded symplectic surface; assume that $\partial F$ is ascending. Then $F$ is a quasipositive surface. 
\end{thm}
\subsubsection{The result}

\begin{prop}\label{egalite si minimal}
	Suppose that $\Sigma_n$'s, $\Sigma_0$, $M$ verify Assumption \ref{basic assumption}. Assume moreover that the $\Sigma_n$'s are minimal and that $k^T+k^N=0$.
	\begin{enumerate}
		\item 
		There exists a parallel complex structure $J_0$ on $T_pM$ such that every tangent plane to $\Sigma_0$ at $p$ is a $J_0$-complex line. 
		\item 
		There exists a symplectic structure $\omega_0$ in a neighbourhood of $p$ such that, for $n$ large enough and $\epsilon$ small enough, the $\Sigma_n^\epsilon$'s are $\omega_0$-symplectic. 
		\item 
		For $n$ large enough, $\epsilon$ small enough, the links $\partial\Sigma_n^\epsilon$'s are quasi-positive and 
		\begin{equation}
		\chi(\Sigma_n^\epsilon)=\chi_s(\partial\Sigma_n^\epsilon)
		\end{equation}
	\end{enumerate}
	
\end{prop}
\begin{proof}  The pseudo-holomorphic curve $S$ described in Theorem \ref{theorem: current in the minimal case} is closed and included in the projective line $(Z_+)_p$; it has zero homology so it consists in a finite number of points. On the other hand, the $\tilde{\Sigma}_n^\epsilon$'s are connected so $\tilde{\Sigma}_0^\epsilon\cup S$, being their Hausdorff limit, is also connected. Since $\tilde{\Sigma}_0^\epsilon$ is closed, it follows that the limit of the $\tilde{\Sigma}_n^\epsilon$'s.\\ 
The set $\tilde{\Sigma}_0^\epsilon$ is the union of the lifts of the branched disks making up $\Sigma_0^\epsilon$ and the intersection $\tilde{\Sigma}_0^\epsilon\cap (Z_+)_p$ is the set of the complex structures on the planes tangent at $p$ to the different disks making up $\Sigma_0$. \\
Suppose that there are two different complex structures, $J_0$ and $J_1$ in\\ $\Sigma_0^\epsilon\cap (Z_+)_p$. There are sequences of points $p_n$ and $q_n$ in $\Sigma_n$, both converging to $p$ such that the lifts $J(p_n)$ and $J(q_n)$ in $Z_+(M)$ converge respectively to $J_0$ and $J_1$. For $n$ large enough, there is a path $\gamma_n$  between $J(p_n)$ and $J(q_n)$ in $\tilde{\Sigma}_n$; the $\gamma_n$'s converges to a path in $\tilde{\Sigma}_0^\epsilon\cap (Z_+)_p$ between $J_0$ and $J_1$; but this latter space is finite, a contradiction. This proves 1.\\
\\	
We denote by $\langle, \rangle_0$ the scalar product on $T_pM$ and we let\begin{equation}
	\omega_0(X,Y)=\langle J_0X, Y\rangle_0
	\end{equation}
	We take a trivialization of $Z_+$ around $p$. The curve $\tilde{\Sigma}_0^\epsilon$ is in a small neighbourhood of $J_0$ and so is $\tilde{\Sigma}_n^\epsilon$ since it Hausdorff converges to $\tilde{\Sigma}_0^\epsilon$. \\
	If $q\in\Sigma_n^\epsilon$ and $(\epsilon_1,\epsilon_2)$ is a positive orthonormal basis of $T_q\Sigma_n^\epsilon$, 
	\begin{equation}
	\langle J_n(q)\epsilon_1,\epsilon_2\rangle=1
	\end{equation}
	where $J_n(q)$ is the lift of $q$ in $\tilde{\Sigma}_n^\epsilon$. Since $J_n(q)$ is close to $J_0$, if $n$ is large enough and $\epsilon$ is small enough, 	$\omega_0(\epsilon_1,\epsilon_2)>0$. This proves 2.\\
	Since the complex structure on the tangent planes to the $\partial\Sigma_n^\epsilon$ are close to $J_0$, near the boundary, $\partial\Sigma_n^\epsilon$ is ascending w.r.t. the contact form defined by $J_0$ and the theorem follows from Theorem \ref{bo}.
\end{proof}
\section{Exemples with $-k^T>|k^N|$}\label{exemples et contre-exemples}
\subsection{Embedded surfaces}
Consider the complex curve $\Gamma_\epsilon$ defined near $(0,0)$ in $\mathbb{C}^2$ by $z_1^3-z_2^2=\epsilon$ which converges to the cusp parametrized by $(z^2,z^3)$. It verifies $k^T=-k^N=-3$. For every $\epsilon$ glue to $\Gamma_\epsilon$ a little handle closer and closer to the origin. This will not change $k^N$ but $k^T$ will become $-5$. 
\subsection{Immersed minimal disks}
 We recall (see for example [S-V3]) that a minimal disk in $\mathbb{R}^4$  is given locally by a map from the disk $\mathbb{D}$ 
$$F:\mathbb{D}\longrightarrow\mathbb{R}^4\cong\mathbb{C}^2$$
\begin{equation}
F:z\mapsto \Big( f_1(z)+\overline{f_2}(z), f_3(z)+\overline{f_4}(z)\Big)
\end{equation}
 where $f_1,...f_4$ are holomorphic functions verifying

\begin{equation}\label{equation:four functions}
f'_1f'_2+f'_3f'_4=0
\end{equation}
After identifying $\mathbb{S}^2$ to $\mathbb{C}P^1$ and taking a stereographic projection, the Gauss maps $\gamma_\pm:\mathbb{D}\longrightarrow Z_\pm$ are 
\begin{equation}\label{applications de Gauss}
\gamma_+=\frac{f'_3}{f'_2}\ \ \ \ \ \ \ \gamma_-=-\frac{f'_4}{f'_2}
\end{equation}
\begin{equation}\label{definition des f'}
\mbox{We pick}\ \ \ \ \ \ \ \ \ \ \ \ f'_1=z^2, f'_2=z^5, f'_3=z^3, f'_4=-z^4 
\end{equation}
and derive a minimal map 
\begin{equation}\label{definition de l'application minimale}
z\mapsto (\frac{1}{3}z^3+\frac{1}{6}\bar{z}^6,\frac{1}{4}z^4-\frac{1}{5}\bar{z}^5)
\end{equation}
Note that the knot of the branch point is a $(3,4)$ torus knot. \\
We define a sequence of minimal immersions converging to (\ref{definition de l'application minimale}) by
\begin{equation}\label{disque convergent}
h^{(n)'}_1=(z-\frac{1}{n})(z+\frac{1}{n}), h^{(n)'}_2=z^5, 
 h^{(n)'}_3=z^2(z-\frac{1}{n}), h^{(n)'}_4=-z^3(z+\frac{1}{n})
\end{equation}
Using (\ref{applications de Gauss}), we see that the Gauss maps of the minimal surfaces defined by (\ref{disque convergent}) are $$\gamma^{(n)}_+=\frac{1}{z^3}(z-\frac{1}{n}) \ \ \ \ \ \ \ \ \gamma^{(n)}_-=\frac{1}{z^2}(z+\frac{1}{n})$$
We notice that
$$\gamma^{(n)}_+(\frac{1}{n})=0 \ \ \ \ \gamma^{(n)}_+(0)=\infty \ \ \ \ \ \ \gamma^{(n)}_-(-\frac{1}{n})=0 \ \ \ \ \gamma^{(n)}_-(0)=\infty$$
So the limiting currents in $Z_+$ and $Z_-$ each contain at least two points; we derive from the proof of Proposition \ref{egalite si minimal} 1. that $k^N+k^T$ and $k^N-k^T$ are both non zero.

\footnotesize{Marina Ville, Univ. Paris Est Creteil, CNRS, LAMA, F-94010 Creteil, France, \\
	villemarina@yahoo.fr}
	\end{document}